  \documentclass[12pt]{amsart}
  
\usepackage{amssymb} 
\usepackage{xcolor} 
\usepackage{fancybox}
\usepackage{dashbox}
\usepackage{mathrsfs}
\newboolean{showcomments}
\setboolean{showcomments}{true}
\ifthenelse{\boolean{showcomments}}
{ \newcommand{\mynote}[3]{
  \fbox{\bfseries\sffamily\scriptsize#1}
  {\small$\blacktriangleright$\textsf{\emph{\color{#3}{#2}}}$\blacktriangleleft$}}}
{ \newcommand{\mynote}[3]{}}
\definecolor{asparagus}{rgb}{0.53, 0.66, 0.42}

  \usepackage{float}

  \usepackage{pgffor}
  \usepackage{tikz}
  \usetikzlibrary{calc,3d,arrows,positioning,shapes.misc}

  \usepackage[english]{babel}
  \usepackage{amsthm}
  \usepackage{amsmath}
  \usepackage{amssymb}
  \usepackage[latin1]{inputenc}
  \usepackage{setspace}
  \usepackage{enumerate}
  \usepackage{stmaryrd}

  \usepackage[colorlinks=true,linkcolor=black,citecolor=black,filecolor=black,urlcolor=black,menucolor=black]{hyperref}

  \usepackage{array}
  \usepackage{graphicx}
  \usepackage[babel]{csquotes}
  \usepackage{geometry}
  \usepackage{bbm}
  \usepackage{stmaryrd}
  \usepackage[all]{xy}
  \usepackage{mathrsfs}

  \theoremstyle{plain}
  \newtheorem{thm}{Theorem}

  \newtheorem{lem}[thm]{Lemma}
  
  \theoremstyle{definition}
  \newtheorem{defn}[thm]{Definition}
  
  \newtheorem{rem}[thm]{Remark}

  \usepackage{xpatch}
  \xpatchcmd{\proof}{\itshape}{\normalfont\bfseries}{}{}

  \newcommand{\step}[1]{\par\medskip\par\noindent\textit{#1}} 

 \def \eps{\varepsilon}
  \def \F{\mathscr{F}}
  \def \I{\mathscr{I}}

  \def \R{\mathbb{R}}
  
  \def \Sym{\mathcal{S}(d)}

  \newcommand {\lip} {\mathop \textup{Lip}}

\begin{document}

\title[Well-posedness for degenerate elliptic PDE]{Well-posedness for degenerate elliptic PDE arising in optimal learning strategies}
\author{Tim Laux}
\address{Department of Mathematics, 
	University of California,
	 Berkeley, California 94720-3840 USA}
\email{tim.laux@math.berkeley.edu}

\author{J.\ Miguel Villas-Boas}
\address{Haas School of Business,
	University of California, Berkeley
	Berkeley, California 94720-1900 USA}
\email{villas@haas.berkeley.edu}

\maketitle
    \begin{abstract}
  	
  	We derive a comparison principle for a degenerate elliptic partial differential equation without boundary conditions which arises naturally in optimal learning strategies. Our argument is direct and exploits the degeneracy of the differential operator to construct (logarithmically) diverging barriers.
    \medskip
    
  \noindent \textbf{Keywords:} Viscosity solutions, degenerate elliptic PDE, comparison principle
  
  \medskip

  \noindent \textbf{Mathematical Subject Classification:} 35D40, 35B50, 35J70, 49L25, 93E35
  \end{abstract}

%
%
%

The purpose of this paper is the analysis of a degenerate elliptic partial differential equation (PDE) arising from a stochastic process in optimal learning strategies. Our main result establishes the uniqueness of viscosity solutions to this degenerate PDE \emph{without} boundary conditions. The PDE we consider, cf.\ \eqref{eq:pde classical} for its precise form, shares key features with the following toy model
\[
	\max
	\Big\{ 
		\max_{1\leq i \leq d}
		\big\{
			-x_i^2(1-x_i)^2 u_{x_ix_i} + u
			\big\}, 
			u- \varphi(x)
	\Big\}
	=0 \quad \text{in } (0,1)^d,
\]
a fully nonlinear degenerate elliptic PDE with obstacle $\varphi$. Note that no boundary conditions are imposed on the faces of the unit cube $(0,1)^d$. 
An excellent framework to study such equations is the theory of viscosity solutions, starting from the fundamental work of Crandall, Ishii, and Lions \cite{crandall1992user}, cf.\ \cite{touzi2012optimal} for an exposition of relevant parts of this field and the connection to stochastic differential equations. However, the additional degeneracy due to the vanishing of $-x_i^2(1-x_i)^2 $ at the boundary and the abscence of boundary conditions causes additional difficulties which we will handle here with care. 
 Our proof is direct and elementary. The key idea is to exploit the degeneracy of the differential operator at the boundary to construct (logarithmically) diverging barriers. 
 On a different note, regarding the regularity of the solution $u$ across the free boundary and the more subtle question of the regularity of the free boundary itself, we refer the interested reader to the crucial work of Caffarelli \cite{Caffarelli1977}, and the expository notes \cite{figalli2018free}. 
 
The model we analyze in this paper was introduced in the recent paper \cite{KeVil17} by Ke and one of the authors: a decision maker may decide among $d$ alternatives on which they can collect information, and ultimately invest. 
In this model, the payoff $\pi_i$ of the $i^{\textup{th}}$ alternative is assumed to take either of the two values $\underline\pi_i < \overline\pi_i$. These random variables $(\pi_1,\dots,\pi_d)$ are assumed to be independent. The outside option has the deterministic payoff $\pi_0$. 
%
%
%
%
As the process $\pi_i$ is assumed to take only two values, it is characterized by the belief $X_i(t) = P(\pi_i(t) =\overline \pi_i | \F_t)$, where $\F_t$ is a filtration representing the observed signals until time $t$. The decision maker's allocation policy $\I = \{I_t\}_{t>0}$ controls the stochastic differential equation (SDE) 
\begin{equation}\label{SDE}
	dX_i = \frac{\overline \pi_i - \underline \pi_i}{\sigma_i^2} X_i \left(1-X_i\right) 
	\left\{ 
		\left[\pi_i- \underline\pi_i (1-X_i) - \overline \pi_i X_i\right]dT_i
		+ \sigma_i dW(T_i)
	\right\},
\end{equation}
where $T_i(t) = |\{ s\in(0,t)\colon I_s(X_i(s))=1\}|$ denotes the accumulated time alternative $i$ has been investigated.

Let us briefly argue why this is sensible intuitively; we refer the interested reader to \cite{KeVil17} for the derivation of \eqref{SDE} and more details: The larger the signal-to-noise ratio $\frac{\overline \pi_i - \underline \pi_i}{\sigma_i^2}$ is, the more likely will the decision maker update their belief according to a new signal. Also note that it makes sense that the right-hand side is decreasing in the noise $\sigma_i$: the larger the noise, the less likely will the decision maker trust the new signal and update their belief. The drift $\pi_i- \underline\pi_i (1-X_i) - \overline \pi_i X_i$ is simply the difference of the prospected and expected value of alternative $i$. Finally, note that the prefactor $X_i(1-X_i)$ indicates that the decision maker is most likely to update their belief if they were undecided in the first place, i.e., if $X_i=\frac12$.

The expected payoff at a stopping time $\tau$ and under an allocation policy $\I$ is
\begin{equation}\label{eq:payoff}
J(x;\I,\tau) := E \Big[
\max\big\{ \max_{1\leq i \leq d} \{ \overline\pi_i X_i + \underline\pi_i(1-X_i)\},   \pi_0\big\}
- \sum_{i=1}^d c_i T_i(\tau)
 \Big| X(0)=x
\Big].
\end{equation}
The decision maker's objective is to maximize the payoff:
\begin{equation}\label{eq:def V}
	V(x) = \sup_{\I,\tau} J(x;\I,\tau).
\end{equation}
In \cite{KeVil17} it is argued that the value function $V$ satisfies the Hamilton-Jacobi-Bellman PDE
\begin{equation}\label{eq:PDE V}
	 \max\left\{\max_{1\leq i \leq d} \left\{ \frac{\left(\overline \pi_i - \underline \pi_i\right)^2}{2\sigma_i^2} x_i^2 (1-x_i)^2 V_{x_ix_i}-c_i  \right\} , g(x)-V\right\}=0 \quad \text{in } Q,
\end{equation}
where $Q=(0,1)^d$ denotes the unit cube and 
$g(x): = 
\max\left\{ 
	 \max_i
	 	\left\{
			\overline\pi_ix_i + \underline\pi_i (1-x_i)
		\right\},
		\pi_0
\right\}
$.

Clearly, the payoff satisfies $\pi_0 \leq J(x;\I,\tau) \leq \max_i \overline \pi_i$ and hence we have the uniform bounds for the value function
\[
0< \pi_0 \leq V(x) \leq \max_{1\leq i \leq d} \overline \pi_i <\infty.
\]
The Lipschitz-continuity of the value function $V \in C^{0,1}( Q)$ can be derived in the following way. Given two points $x,x'\in Q$, using the optimal allocation policy $\I$ and stopping time $\tau$ of $x$ for $x'$ yields
\[ 
	V(x)-V(x') \leq J(x;\I,\tau) -J(x';\I,\tau) 
	= E \big[
	f(X_\tau)
	\big| X_0=x
	\big]
	- E \big[
	f(X_\tau)
	\big| X_0=x'
	\big],
\]
where we momentarily defined the function $f(x):=\max\big\{ \max_{1\leq i \leq d} \{ \overline\pi_i x_i +\underline\pi_i(1-x_i) \},   \pi_0\big\}$, which is Lipschitz continuous as the maximum of linear functions. Note that due to the choice of the allocation policy, the last term in $J$, which is nonlocal in time cancels exactly.
Hence, denoting by $\lip(f)$ the Lipschitz constant of $f$, 
\[
	V(x)-V(x') \leq  \lip(f) E\big[ |X_\tau-X'_\tau| \big| X_0 = x, X'_0 =x'\big],
\]
so that we only need to appeal to the stability of the SDE \eqref{SDE} to obtain $V(x)-V(x') \leq C |x-x'|.$ Interchanging the roles of $x$ and $x'$ proves the Lipschitz continuity of $V$.


\medskip


Let us suppose for simplicity that all payoffs and noise levels are equal so that \eqref{eq:PDE V} becomes
\begin{equation}\label{eq:PDE easyV}
	\max\left\{\max_{1\leq i \leq d} \left\{ x_i^2 (1-x_i)^2 V_{x_ix_i}-c_i  \right\} , g(x)-V\right\}=0 \quad \text{in } Q.
\end{equation}

In order to rewrite \eqref{eq:PDE easyV} in a more familiar  form we use the change of variables $V(x)=b-e^{ u(x)}$, where $b> \max_i \overline \pi_i$, which leads us to
\begin{equation}\label{eq:pde classical}
 \max\left\{\max_{1\leq i \leq d} \left\{- x_i^2 (1-x_i)^2  \left(  u_{x_ix_i} -  u_{x_i}^2 \right)-c_i e^{- u} \right\} , 1-(b-g(x)) \,e^{-u}\right\}=0 \quad \text{in } Q,
\end{equation}
where $g$ could now be any given continuous function on $\bar Q$ such that $b-g(x)>0$ for all $x\in \bar Q$ and $c_1,\dots,c_d>0$ are given positive constants. Note that the condition on $g$ is true in the above concrete example since $b> \max_i  \overline \pi_i$

Setting
\[
F(x,r,p,A) := \max\left\{\max_{1\leq i \leq d} \left\{-x_i^2 (1-x_i)^2  \left( A_{ii} -  p_i^2 \right)-c_i e^{-   r} \right\} , 1-(b-g(x)) \,e^{-r}\right\}
\]
for $x\in Q$, $r\in \R$, $p\in \R^d$, and $A=(A_{ij})_{i,j=1}^d \in \Sym = \{M\in \R^{d\times d} \colon M^T=M\}$, the equation reads
\[
F(x,u,Du,D^2u)=0,
\]
which can be formulated in the viscosity sense, see Definition \ref{def:viscosity_solution} below.


The main result of this work is the following comparison theorem.

\begin{thm}\label{thm:main}
If $u,v\in C^{0,1}(Q)$ satisfy 
\begin{equation}\label{eq:subsolution}
F(x,u,Du,D^2u) \leq 0 \quad \text{and} \quad F(x,v,Dv,D^2v) \geq 0 \quad \text{in the viscosity sense,}
\end{equation}
then $u\leq v$ in $Q$.
\end{thm}

\begin{rem}
As an immediate consequence, the viscosity solution of \eqref{eq:pde classical}, without boundary conditions, is unique in the class $C^{0,1}(Q)$.  Since the construction for $V$ is of this class and bounded, also $u(x)=\log(b-V(x))$ is Lipschitz and the problem is well-posed.
\end{rem}

\begin{rem}
There are two obvious difficulties: 
\begin{enumerate}
\item There are no boundary conditions.
\item $F$ is not uniformly elliptic.
\end{enumerate} 
We will see that these two properties are interconnected: The degeneracy at the boundary makes boundary conditions oblivious. Indeed, the underlying stochastic process \eqref{SDE} does not reach the boundary $\partial Q$ in finite time. This is in fact due to the degeneracy: even at the simpler example $dX_t = X_t dW_t$ one can see this effect by a direct computation.

Since $c_i>0$ and $g(x) < b$, the function $F$ is strictly monotonic increasing in $r$:
For every $R>0$ there exists a constant $\theta>0$ such that for all $x\in Q$, $p \in \R^d$ and $A \in \Sym$ we have
\begin{equation}\label{eq:monotonicity_classical}
 \theta( r-r') \leq F(x,r,p,A) - F(x,r',p,A)  \quad \text{for all } r'\leq r \leq R.
\end{equation}
\end{rem}
For the rest of the paper we will assume for simplicity that $c_1=\ldots=c_d = 1 $.

\medskip

Let us first recall the definition of viscosity solutions.

\begin{defn}\label{def:viscosity_solution}
A continuous function $u\in C(Q)$ is called a subsolution (supersolution) of $F(x,u,Du,D^2u)=0$ and we write
\begin{equation}\label{eq:defFleq0}
F(x,u,Du,D^2u) \leq  0 \;(\geq0) \quad \text{in the viscosity sense}
\end{equation}
if the following holds: 
Let $\zeta\in C^2(Q)$ be such that $\zeta-u$ has a local maximum (minimum) at $x_0$, then 
\begin{equation}\label{eq:defFgeq0}
F(x,\zeta,D\zeta,D^2\zeta)\leq0 \;  (\geq0) 
\end{equation}
at the point $x=x_0$.
The function $u$ is called a solution of $F(x,u,Du,D^2u)=0$ and we write 
\[
F(x,u,Du,D^2u) =0 \quad \text{in the viscosity sense}
\]
if $u$ is both a sub- and a supersolution.
\end{defn}


\begin{lem}
Let $u\in C(Q)$. Then 
\[
F(x,u,Du,D^2u) \leq 0 \quad \text{in the viscosity sense}
\]
 if and only if for every $x_0\in Q$ and any generalized superjet $(p,A) \in \bar J_+ u\, (x_0)$ 
\[
F(x_0,u_0,p,A)\leq 0.
\]
Let $v\in C(Q)$. Then 
\[
F(y,v,Dv,D^2v) \geq 0 \quad \text{in the viscosity sense}
\]
 if and only if for every $y_0\in Q$ and any generalized subjet $(p,A) \in \bar J_- v\, (y_0)$ 
\[
F(y_0,v(y_0),p,A)\geq 0.
\]
\end{lem}
Here and in the following, by $\bar J_+ u\, (x_0)$, $\bar J_- v\, (y_0)$ we denote as usual the set of all generalized super- and subjets, respectively:
\begin{enumerate}
\item  For $(p,A)\in \R^d \times \Sym $, we have $(p,A)\in J_+ u(x_0)$ if and only if $(p,A)$ is a superjet of $u$ at $x_0$, i.e., 
\[
 u(x) \leq u(x_0) + p\cdot \left(x-x_0\right) + \frac12\left(x-x_0\right) \cdot A \left(x-x_0\right) + o\left(\left|x-x_0\right|^2\right) \quad \text{as }x\to x_0.
\]
The set of all generalized superjets $\bar J_+ u(x_0)$ is simply given by the topological closure of this set.
\item  Similarly, for $(p,A)\in \R^d \times \Sym $, we have $(p,A)\in J_- v(y_0)$ if and only if $(p,A)$ is a subjet of $v$ at $y_0$, i.e., 
\[
 v(y) \geq v(y_0) + p\cdot \left(y-y_0\right) + \frac12\left(y-y_0\right) \cdot A \left(y-y_0\right) + o\left(\left|y-y_0\right|^2\right) \quad \text{as }y\to y_0.
\]
The set of all generalized subjets $\bar J_- v(y_0)$ is the closure of $J_- v(y_0)$.
\end{enumerate}


\medskip


We first state and prove some lemmas which will be useful for the proof of the theorem.
The first lemma exploits the degeneracy at the boundary by adding a penalization- or barrier-term which diverges at the boundary.
Note that the barriers $u_\eps \leq u$ and $v^\eps \geq v$ diverge as we approach the boundary $\partial Q$.

\begin{lem}\label{lem:log}
Let $\alpha,\eps>0$. 
For $u,v\in C^{0,1}(Q)$ let
\begin{align}\label{eq:def_ueps}
u_\eps (x) := u(x) - \eps \,\Phi(x), \quad
v^\eps (y) := v(y) + \eps \, \Phi(y),
\end{align}
where 
\begin{equation}\label{eq:def_Phi}
\Phi(x) := -\sum_{i=1}^d \left\{ \log x_i + \log(1-x_i)\right\} \geq 0,
\end{equation}
and let $x_{\alpha,\eps}, y_{\alpha,\eps}$ be such that 
\begin{equation}\label{eq:xymax}
u_\eps(x_{\alpha,\eps})-v^\eps(y_{\alpha,\eps}) - \frac\alpha2 \left| x_{\alpha,\eps}-y_{\alpha,\eps}\right|^2 
=\sup_{(x,y)\in Q\times Q} \left\{ u_\eps(x)-v^\eps(y) - \frac\alpha2 \left| x-y\right|^2  \right\}.
\end{equation}
Then 
\[
\alpha  \left| x_{\alpha,\eps}-y_{\alpha,\eps}\right|
\leq 2 \max\left\{ \lip(u), \lip(v)\right\}.
\]
\end{lem}

Here $\lip(u) := \sup_{x,y\in Q} \frac{|u(x)-u(y)|}{|x-y|}$ denotes the Lipschitz constant of the function $u$ on $Q$.


\begin{proof}
Testing the maximality \eqref{eq:xymax} with the pair $(x_{\alpha,\eps},x_{\alpha,\eps})$ yields
\[
u_\eps(x_{\alpha,\eps})-v^\eps(y_{\alpha,\eps}) - \frac\alpha2 \left| x_{\alpha,\eps}-y_{\alpha,\eps}\right|^2
 \geq u_\eps(x_{\alpha,\eps})-v^\eps(x_{\alpha,\eps}) 
\]
and hence after reordering and using the definition of $v^\eps$
\[
\frac\alpha2 \left| x_{\alpha,\eps}-y_{\alpha,\eps}\right|^2 \leq v^\eps(x_{\alpha,\eps}) -v^\eps(y_{\alpha,\eps}) 
\leq \lip(v) \left| x_{\alpha,\eps}-y_{\alpha,\eps}\right| +\eps \left( \Phi(x_{\alpha,\eps})-
\Phi(y_{\alpha,\eps}) \right).
\]
However, we have no control over $\Phi(x_{\alpha,\eps}) -\Phi(y_{\alpha,\eps})$.
To overcome this difficulty, we test the maximality as well with $(y_{\alpha,\eps},y_{\alpha,\eps})$ instead of $(x_{\alpha,\eps},x_{\alpha,\eps})$ and obtain
\[
u_\eps(x_{\alpha,\eps})-v^\eps(y_{\alpha,\eps}) - \frac\alpha2 \left| x_{\alpha,\eps}-y_{\alpha,\eps}\right|^2
 \geq u_\eps(y_{\alpha,\eps})-v^\eps(y_{\alpha,\eps})
\]
which yields
\[
\frac\alpha2 \left| x_{\alpha,\eps}-y_{\alpha,\eps}\right|^2  \leq
u_\eps(x_{\alpha,\eps}) - u_\eps(y_{\alpha,\eps})
\leq \lip(u) \left| x_{\alpha,\eps}-y_{\alpha,\eps}\right| -\eps \left( \Phi(x_{\alpha,\eps})-
\Phi(y_{\alpha,\eps}) \right).
\]
Since either $\Phi(x_{\alpha,\eps}) -\Phi(y_{\alpha,\eps}) \geq0$ or $\Phi(x_{\alpha,\eps}) -\Phi(y_{\alpha,\eps})<0$, we may use our favorite among the two above estimates which yields
\[
\frac\alpha2 \left| x_{\alpha,\eps}-y_{\alpha,\eps}\right|^2 \leq
\max\left\{\lip(u) ,\lip(v)\right\} \left| x_{\alpha,\eps}-y_{\alpha,\eps}\right|. \qedhere
\]
\end{proof}


Clearly, modifying a strong solution $u$ of $F(x,u,Du,D^2u)=0$ by the logarithmic barrier $ \Phi$ given in \eqref{eq:def_Phi} yields again a solution $u^\eps=u-\eps \,\Phi$ to some modified equation $F_\eps(x,u_\eps,Du_\eps, D^2u_\eps)=0$. The following lemma states that this is also true in the case of viscosity solutions. 
Although the modified functions $u_\eps $ and $v^\eps$ diverge at the boundary, thanks to the degeneracy, the modified degenerate elliptic operators $F_\eps$ and $F_{-\eps}$ are well-behaved close to the boundary.

\begin{lem}\label{lem:epsviscositysolution}
Let $u,v \in C^0(Q)$ satisfy \eqref{eq:subsolution}.
Then $u_\eps $ and $v^\eps$ defined via \eqref{eq:def_ueps} satisfy
\begin{equation}\label{eq:subsolution}
F_\eps(x,Du_\eps,D^2u_\eps; u) \leq 0 \quad \text{and} \quad F_{-\eps}(x,Dv^\eps,D^2v^\eps; v) \geq 0 \quad \text{in the viscosity sense,}
\end{equation}
where for $x\in Q$, $p\in \R^d$, $A\in \Sym$ and $u\colon Q\to \R$,
\begin{equation}\label{eq:def_Feps}
\begin{split}
F_\eps(x,p,A; u):= 
\max\Big\{\max_{1\leq i \leq d} \big\{-(x_i^2 (1-x_i)^2  \left( A_{ii} -  p_i^2 \right)- e^{- u(x)} 
+&f_{\eps}(x_i,p_i)\big\} , \\
&1-(b-g(x)) e^{-u(x)}\Big\}
\end{split}
\end{equation}
and
\[
f_{\eps}(x_i,p_i):=-\eps \left((1-x_i)^2 +x_i^2 \right) 
+ \eps^2 \left(1-2x_i\right)^2
+2\eps \,x_i \left( 1-x_i \right)\left(2x_i-1\right)  p_i
\]
\end{lem}

Note that the $u$-variable in $F_\eps$ is frozen in the sense that we plug in the fixed function $u$, not the candidate $u_\eps$. On the one hand, then the terms containing $u$ become simply another $x$-dependence in the equation. On the other hand, we want to keep track of the frozen variable to remember the crucial strict monotonicity \eqref{eq:monotonicity_classical} in that variable.
Note furthermore that the $u$-dependence is only pointwise.


\begin{proof}
We only show the statement for $u$ as the one for $v$ is completely analogous.
Let $x_0\in Q$ be fixed.

 Note that there is a one-to-one correspondence between $J_+u_\eps(x_0)$ and $J_+ u(x_0)$: Let $(p_\eps, A_\eps) \in \bar J_+ u_\eps(x_0)$, i.e.,
\begin{align*}
u(x)-\eps\Phi(x)
\leq& u(x_0) - \eps \Phi(x_0)+ p_\eps^T \left(x-x_0\right) + \frac12 \left(x-x_0\right)^T A_\eps \left(x-x_0\right) +o\left(\left|x-x_0\right|^2\right).
\end{align*}
Developing $\Phi$ to second order around $x_0$
\begin{align*}
u(x) \leq u(x_0) + \left(p_\eps^T+\eps D\Phi(x_0)\right) \left(x-x_0\right) +\frac12 \left(x-x_0\right)^T \left(A_\eps+\eps D^2\Phi(x_0) \right)&\left(x-x_0\right)\\
 &+o\left(\left|x-x_0\right|^2\right),
\end{align*}
i.e.,
\[
(p_\eps+\eps D\Phi(x_0)^T, A_\eps+\eps D^2\Phi(x_0) ) \in J_+ u(x_0).
\]

Therefore, if $(p_\eps, A_\eps) \in  \bar J_+ u_\eps(x_0)$ is a generalized superjet of $u$ at $x_0$, then using the above argument for an approximating sequence, we obtain
\[
(p_\eps+\eps D\Phi(x_0)^T, A_\eps+\eps D^2\Phi(x_0) ) \in \bar J_+ u(x_0).
\]
Using the fact that $u$ is a subsolution, i.e., \eqref{eq:subsolution}, we obtain
\[
F\left(x_0,u,p_\eps+\eps D\Phi(x_0)^T, A_\eps+\eps D^2\Phi(x_0)\right) \leq 0.
\]
That means that at the point $x=x_0$ we have
\begin{align*}
 &\max\Big\{\max_{1\leq i \leq d} \big\{-x_i^2 (1-x_i)^2  \left( A_{\eps,ii} -  p_{\eps,i}^2 \right) + \\
 &+ x_i^2\left(1-x_i\right)^2 \left( - \eps\Phi_{x_ix_i}+ 2 \eps \Phi_{x_i}p_{\eps,i} + \eps^2 \Phi^2_{x_i} \right)
 - e^{- u(x)} \big\} ,  1-(b-g(x)) e^{-u(x)}&\Big\} \leq 0.
\end{align*}
As $\Phi_{x_i}=-\frac1{x_i} + \frac1{1-x_i}$ and $\Phi_{x_ix_i} = \frac1{x_i^2} + \frac1{(1-x_i)^2}$, this is nothing else but \eqref{eq:def_Feps}.
\end{proof}

One important ingredient of the proof is the by now classical doubling of variables first introduced by Jensen and then formulated by Ishii in (a more general form than) the following lemma.

\begin{lem}[Ishii's Lemma, see Theorem 3.2 in \cite{crandall1992user}]\label{lem:Ishii}
Let $u, v\in C^0(Q)$, $\alpha>0$ and suppose 
\[
u(x_0)-v(x_0)-\frac\alpha2 \left|x_0-y_0\right|^2 = 
\max_{(x,y)\in Q\times Q} \left\{ u(x)-v(x)-\frac\alpha2 \left|x-y\right|^2 \right\}
\]
Then there exist $A,B\in \Sym$ such that
\begin{equation}\label{eq:Ishiigradients}
\left(\alpha(x_0-y_0), A \right) \in \bar J_Q^+ u\, (x_0), 
\left(\alpha(x_0-y_0), B \right) \in \bar J_Q^- v\, (y_0), 
\end{equation}
and 
\begin{equation}\label{eq:Ishiiineq}
-3\alpha 
\begin{pmatrix}
I_d&0\\
0& I_d
\end{pmatrix}
\leq 
\begin{pmatrix}
A&0\\
0&-B
\end{pmatrix}
\leq
3\alpha
\begin{pmatrix}
I_d& -I_d\\
-I_d& I_d
\end{pmatrix}.
\end{equation}
\end{lem}

Here $I_d$ and $0$ denote the identity and zero matrix in $\R^d$, respectively.
The inequalities in \eqref{eq:Ishiiineq} are to be understood in the sense of symmetric matrices (or equivalently symmetric bilinear forms).


\begin{proof}[Proof of Theorem \ref{thm:main}]
Suppose for a contradiction that 
\[
0< \delta = (u-v)(z)\quad \text{for some } z\in Q.
\]
Let $\alpha,\eps>0$ be fixed and let $u_\eps$ and $v^\eps $ be given by $\eqref{eq:def_ueps}$. (One may think of $\alpha \gg 1$ and $\eps \ll 1$.)

\step{Step 1:}  By Bolzano-Weierstrass, the supremum
\[
M_{\alpha,\eps} := \sup_{(x,y)\in Q\times Q} \left\{ u_\eps(x)-v^\eps(y) - \frac\alpha2\left|x-y\right|^2  \right\}
\]
is attained at some interior point $(x_{\alpha,\eps}, y_{\alpha,\eps}) \in Q\times Q$ since $ u_\eps(x) \to - \infty$ as $x\to \partial Q$ and $v^\eps (y) \to +\infty $ as $y\to \partial Q$.
Since $\bar Q$ is compact, there exist sequences $\alpha_n \uparrow \infty$, $\eps_n\downarrow 0$ such that 
\begin{align*}
x_n &:= x_{\alpha_n,\eps_n} \to \bar x \in \bar Q,\\
y_n &:= y_{\alpha_n,\eps_n} \to \bar y \in \bar Q.
\end{align*}
Lemma \ref{lem:log} yields that
\begin{align}\label{eq:alphap_bounded}
\alpha_n\left|x_n-y_n\right| \text{ stays bounded as }n\to \infty
\end{align}
and in particular, since $\alpha_n\to \infty$, this implies $\left|x_n-y_n\right| \to 0$, i.e., $\bar x = \bar y$.

\step{Step 2:} Since the maximizer $(x_n, y_n)$ of $M_n$ is an interior point,  Ishii's Lemma \ref{lem:Ishii} furnishes the existence of two symmetric matrices $A_n, B_n \in \Sym$ which together with $\alpha_n(x_n-y_n)$  contribute second-order sub- and superjets for $u$ and $v$, respectively:
\[
\left(\alpha_n(x_n-y_n), A_n \right) \in \bar J_+ u(x_n),\quad 
\left( \alpha_n(x_n-y_n), B_n \right) \in \bar J_- v(y_n)
\]
and furthermore inequality \eqref{eq:Ishiiineq} holds.

\step{Step 3:}
It is straightforward to see that $F_\eps$ is strictly increasing in the frozen $u$-variable: For every $R>0$ there exists $\theta>0$ such that for all $r'\leq r\leq R$
\begin{equation}\label{eq:monotonicity}
 \theta \left(r-r'\right) \leq F_\eps(x,p,A;r) - F_\eps(x,p,A;r').
\end{equation}
Recall that the dependence of $F_\eps$ on $u$ is pointwise.
Hence by the maximality of $(x_n,y_n)$
\begin{align*}
\theta \delta 
&\stackrel{\phantom{\eqref{eq:monotonicity}}}{\leq} \theta \left(u(x_n)-v(y_n)\right)\\
&\stackrel{\eqref{eq:monotonicity}}{\leq} F_{\eps_n}( x_n,\alpha_n(x_n-y_n), A_n; u(x_n)) 
- F_{\eps_n}( x_n,\alpha_n(x_n-y_n), A_n; v(y_n)).
\end{align*}
By Lemma \ref{lem:epsviscositysolution}, the modified functions $u_\eps$ and $v^\eps$ are sub- and supersolutions (of the modified operators), respectively, so we obtain
\[
F_{\eps_n}(x_n,\alpha_n(x_n-y_n),A_n;u) 
\leq 0 \leq F_{-\eps_n}(y_n,\alpha_n(x_n-y_n), B_n; v).
\]
Combining the two above inequalities yields
\begin{align}\label{eq:tdelta<deltaF}
0<\theta \delta &\leq F_{-\eps_n}(y_n,\alpha_n(x_n-y_n),B_n;v(y_n))
-F_{\eps_n} (x_n,\alpha_n(x_n-y_n),A_n; v(y_n)).
\end{align}

\step{Step 4:} We claim that
\begin{equation}\label{eq:FepsAB}
\begin{split}
&F_{-\eps_n}(y_n,\alpha_n(x_n-y_n),B_n;v(y_n))
-F_{\eps_n} (x_n,\alpha_n(x_n-y_n),A_n; v(y_n))\\
&\;\leq 3\alpha_n\left|x_n-y_n\right|^2 +  \alpha_n^2 \left| x_n-y_n\right|^3 
+8 \eps_n \left( 1+ \alpha_n\left|x_n-y_n\right| + \eps_n\right)
+ \omega_g(|x_n-y_n|),
\end{split}
\end{equation}
which will conclude the proof of the theorem. Here $\omega_g$ denotes the modulus of continuity of the given function $g$.
 Indeed, if the claim is true, then by \eqref{eq:alphap_bounded} we obtain
\begin{align*}
\limsup_{n\to\infty} F^{\eps_n}(y_n,\alpha_n(x_n-y_n),B_n;v(y_n))
-F_{\eps_n} (x_n,\alpha_n(x_n-y_n),A_n; v(y_n)) \leq 0,
\end{align*}
a contradiction to the strict positivity \eqref{eq:tdelta<deltaF}.

We are left with proving \eqref{eq:FepsAB}.
To this end we will use the second inequality in \eqref{eq:Ishiiineq}, which simply means 
\begin{equation}\label{eq:AleqBsimple}
\xi^T A \xi - \eta^T B\eta \leq 3\alpha \left|\xi -\eta\right|^2 \quad \text{for all } \xi, \eta \in \R^d.
\end{equation}

In order to prove \eqref{eq:FepsAB}, let $\eps>0$, $x,y\in Q$, $p\in \R^d$, $A,B\in \Sym$ be given s.t.\ \eqref{eq:AleqBsimple} holds. For notational simplicity set $\sigma_i(x):= \left(1-x_i\right)x_i$. 

We distinguish two cases.
\step{Case 1: $F_{-\eps}(y,p,B;r) = 1-(b-g(y)) e^{-r}$.}

In this case
\[
F_{-\eps}(y,p,B;r) - F_\eps(x,p,A;r) \leq 1-(b-g(y)) e^{-r}- \big(1-(b-g(x)) e^{-r} \big) \leq \omega_g(|x-y|) e^{-r}.
\]
\step{Case 2: $F_{-\eps}(y,p,B;r) > 1-(b-g(y)) e^{-r}$.}

By definition of $F_{-\eps}$, there exists an index $j\in \{1,\dots,d\}$ such that
\[
F_{-\eps}(y,p,B;r) = - \left(\sigma_j(y) e_j\right)^T B\,  \sigma_j(y) e_j + \sigma_j^2(y) p_j^2
-e^{-r} +f_{-\eps}(y_j,p_j).
\]
Hence
\begin{align*}
F_{-\eps}(y,p,B;r) &-F_{\eps}(x,p,A;r)\\
\leq & \left(\sigma_j(x)e_j\right)^T A \,\sigma_j(x)e_j - \left(\sigma_j(y)e_j\right)^T B\, \sigma_j(y)e_j  \\
& + \left(\sigma_j^2(y) - \sigma_j^2(x)\right)p_j^2 + f_{-\eps}(y_j,p_j)-f_\eps(x_j,p_j).
\end{align*}
Applying \eqref{eq:AleqBsimple} to the first right-hand side term (with the collinear vectors $\xi= \sigma_j(x)e_j$ and $\eta=\sigma_j(y)e_j$), we obtain
\begin{align*}
F_{-\eps}(y,p,B;r) -F_{\eps}(x,p,A;r)
\leq &3\alpha \left( \sigma_j(x)- \sigma_j(y)\right)^2 
+\left|\sigma^2_j(x)-\sigma^2_j(y)\right| |p|^2\\
&+ \sup_{t\in (0,1)} \big( \left|f_{-\eps}(t,p_j) \right|+\left|f_\eps(t,p_j) \right|\big).
\end{align*}
Since the gradient $D\sigma_j(x)=\left(1-2x_j\right) e_j^T$ is bounded by $1$ and $0\leq \sigma_j(x) \leq 1$, we obtain 
$
\left|\sigma_j(x)-\sigma_j(y)\right| \leq \left|x-y\right|
$
and 
$
\left|\sigma^2_j(x)-\sigma^2_j(y)\right|  \leq \frac14 \left|x-y\right|.
$
Furthermore,
\[
 \sup_{t\in (0,1)} \big( \left|f_{-\eps}(t,p_j) \right|+\left|f_\eps(t,p_j) \right|\big) \leq 2\eps \left( 2+2|p|+\eps\right).
\]
and therefore \eqref{eq:FepsAB} holds. This concludes the proof of Theorem \ref{thm:main}.
\end{proof}

\section*{Acknowledgement}

The authors would like to thank Craig Evans who made this project possible and generously shared his ideas.

\bibliographystyle{acm}
\bibliography{lit}

\begin{thebibliography}{1}

\bibitem{Caffarelli1977}
{\sc Caffarelli, L.~A.}
\newblock The regularity of free boundaries in higher dimensions.
\newblock {\em Acta Mathematica 139}, 1 (Dec 1977), 155--184.

\bibitem{crandall1992user}
{\sc Crandall, M.~G., Ishii, H., and Lions, P.-L.}
\newblock User's guide to viscosity solutions of second order partial
  differential equations.
\newblock {\em Bulletin of the American mathematical society 27}, 1 (1992),
  1--67.

\bibitem{figalli2018free}
{\sc Figalli, A.}
\newblock Free boundary regularity in obstacle problems.
\newblock {\em arXiv preprint arXiv:1807.01193\/} (2018).

\bibitem{KeVil17}
{\sc Ke, T., and Villas-Boas, J.~M.}
\newblock Optimal learning before choice.
\newblock Sloan Research Paper 5178-16, MIT, 2018.
\newblock To appear in Journal of Economic Theory.

\bibitem{touzi2012optimal}
{\sc Touzi, N.}
\newblock {\em Optimal stochastic control, stochastic target problems, and
  backward SDE}, vol.~29.
\newblock Springer Science \& Business Media, 2012.

\end{thebibliography}

\end{document}